\documentclass{amsart}
\usepackage{amssymb,a4wide}
\usepackage{verbatim}
\usepackage[matrix,arrow]{xy}

\newtheorem{theorem}{Theorem}[section]
\newtheorem{lemma}[theorem]{Lemma}

\newtheorem{proposition}[theorem]{Proposition}

\def\GL{\mathrm{GL}}
\def\Mat{\mathrm{Mat}}
\def\Aut{\mathrm{Aut}}
\def\Ad{\mathrm{Ad}}
\def\End{\mathrm{End}}
\def\Ker{\mathrm{Ker}}
\def\C{\mathbf{C}}
\def\R{\mathbf{R}}
\def\Z{\mathbf{Z}}

\def\Q{\mathbf{Q}}
\def\C{\mathbf{C}}
\def\N{\mathbf{N}}
\def\kk{\mathbf{k}}

\def\onto{\twoheadrightarrow}
\def\la{\lambda}
\def\into{\hookrightarrow}
\def\onto{\twoheadrightarrow}

\begin{document}
\centerline{}

\title{Residual nilpotence for generalizations of pure braid groups}
\author[I.~Marin]{Ivan Marin}
\address{IMJ, Universit\'e Paris VII, 175 rue du Chevaleret, 75013 Paris, France}
\email{marin@math.jussieu.fr}
\date{June 14,2011}

\subjclass[2010]{Primary 20F36; Secondary 20F55}
\medskip

\begin{abstract}
It is known that the pure braid groups are residually torsion-free nilpotent. This property
is however widely open for the most obvious generalizations of these groups, like pure
Artin groups and like fundamental groups of hyperplane complements (even reflection ones).
In this paper we relate this problem to the faithfulness of linear representations,
and prove the residual torsion-free nilpotence for a few other groups.
\end{abstract}

\maketitle

\section{Introduction}

It has been known for a long time (see \cite{FALK,FALKDEUX}) that the pure braid groups are residually nilpotent, meaning that
they have `enough' nilpotent quotients to distinguish their elements, or equivalently that the intersection of their
descending central series is trivial. Recall that a group $G$ is called residually $\mathcal{F}$ for $\mathcal{F}$ a class
of groups if for all $g \in G \setminus \{ 1 \}$ there exists $\pi : G \onto Q$ with $Q \in \mathcal{F}$ such that $\pi(g) \neq 1$. It is also known that they have the far stronger property of
being residually torsion-free nilpotent.  The strongness of this latter assumption is illustrated by the following implications
(where `residually $p$' corresponds to the class of $p$-groups).
$$
\mbox{residually free} \Rightarrow 
\mbox{residually torsion-free nilpotent}  \Rightarrow 
\mbox{residually $p$ for all $p$} $$
{}$$\Rightarrow 
\mbox{residually $p$ for some $p$} \Rightarrow 
\mbox{residually nilpotent} \Rightarrow 
\mbox{residually finite}
$$
Pure braid groups are not residually free. The following proof of this fact
has been communicated to me several years ago by Luis Paris (note however that the pure braid group
on 3 strands $P_3 \simeq F_2 \times \Z$ is residually free;
it has been announced this year that  $P_4$ is not residually free, see \cite{CFR}).

\begin{proposition} The pure braid group $P_n$ is not residually free for $n \geq 5$.
\end{proposition}
\begin{proof}
It is sufficient to show that $P_5$ is not residually free.
Letting $\sigma_1,\dots,\sigma_{4}$ denote the Artin generators of the braid group $B_5$,
$P_5$ contains the subgroup $H$ generated by $a=\sigma_1^2$, $b=\sigma_2^2$, $c=\sigma_3^2$, $d=\sigma_4^2$.
As shown in \cite{DLS} (see also \cite{COLLINS}) this group is a right-angled Artin group, which contains a subgroup $H_0 = <a,b,d>$ isomorphic to $F_2 \times \Z$,
which is well-known to be residually free but not fully residually free (see \cite{BAUMSLAG}). 

If we can exhibit $x \in P_5$
such that the subgroup generated by $x$ and $H_0$ is a free product $\Z * H_0$,
then, by a result of \cite{BAUMSLAG} which states that the free product of two non-trivial groups can be residually free only if
the two of them are fully residually free, this proves that $P_5$ is not residually free.

One can take $x = c b c^{-1}$. 
Indeed, if $<x , H_0>$ were not a free product, then it would exist a word with trivial image of the form $c b^{u_1} c^{-1} y_1 c b^{u_2} c^{-1} y_2 \dots c b^{u_r}c^{-1} y_r$
with $y_i \in < a,b,d >$, $y_i \neq 1$ for $i < r$, $u_i \neq 0$ for $i \leq r$, and $r \geq 1$.
But in a right-angled Artin group generated by a set $X$ of letters, an expression can be reduced if and only if it contains a word
of the form $x \dots x^{-1}$ or $x^{-1} \dots x$ with $x \in X$, such that all the letters in $\dots$ commute with $x$ (see e.g. \cite{SERVATIUS}).
From this it is straightforward to check that the former expression cannot be reduced, and this proves the claim.
\end{proof}

The original approach for proving this property of residual torsion-free nilpotence seems to fail for most of the usual generalizations of
pure braid groups. Another approach has been used in \cite{RESNIL,KRAMINF}, using faithful linear representations,
thus relating the linearity problem with this one. The main lemma is the following one.

\begin{lemma} \label{lemresnil}
Let $N \geq 1$, $\kk$ a field of characteristic $0$ and $A = \kk[[h]]$ the ring of formal power series.
Then the group $\GL_N^0(A) = \{ X \in \GL_N(A) \ | \ X \equiv 1 \mod h \} = 1 + h \Mat_N(A)$ is
residually torsion-free nilpotent.
\end{lemma}
\begin{proof}
Let $G = \GL_N^0(A)$, $G_r = \{ g \in G \ | \ g \equiv 1 \mod h^r \}$,
$$G^{(r)} = \{ X \in \GL_N(\kk[h]/h^r) \ | \ X \equiv 1 \mod h \} = 1 + h \Mat_N(\kk[h]/h^r) \subset
\Mat_N(\kk[h]/h^r)$$
Clearly $G_r \vartriangleleft G$ and the natural map $G \to G^{(r)}$
has for kernel $G_r$, hence $G/G_r$ is isomorphic to a subgroup of $G^{(r)}$.
This latter group is clearly nilpotent, as $(1+h^ux,1+h^vy) \equiv  1 + h^{uv} (xy-yx) \mod h^{uv+1}$
(where $(a,b) = aba^{-1}b^{-1}$),
and torsion-free as $(1+hx)^n \equiv 1 + n hx \mod h^2$ and $\kk$ has characteristic $0$. Thus all the $G/G_r$
are torsion-free nilpotent, and since clearly $\bigcap_r G_r = \{ 1 \}$ we get that $G$ is
residually torsion-free nilpotent.
\end{proof}

Usually, linear representations have their image
in such a group
when they appear
as the monodromy of a flat connection on a \emph{trivial} vector bundle (see \cite{KRAMINF}). However,
we show how to (partly conjecturally) use this approach in situation where this
geometric motivation is far less obvious. In particular, we prove the
following.

\begin{theorem} \label{theoA} If $B$ is an Artin group for which the Paris representation 
is faithful, then its pure subgroup $P$ is residually torsion-free nilpotent.
\end{theorem}

So far, this Paris representation, which is a generalization of the Krammer representation of \cite{KRAM}, has been shown to be faithful only for the case
where $W$ is a finite Coxeter group. By contrast, in the case of the pure braid groups of complex
reflexion groups, which are other natural generalization of pure braid groups, and for which a natural
and possibly faithful monodromy representation has been constructed in \cite{KRAMCRG},
we get the following more modest but unconditional result. 


\begin{theorem} \label{theoB} If $B$ is the braid group of a complex reflection group
of type $G_{25}$, $G_{26}$, $G_{32}$, $G_{31}$, then its pure braid
group $P$ is residually torsion-free nilpotent.
\end{theorem}

The pure braid groups involved in the latter statement are equivalently
described as the fundamental groups of complements of
remarkable configurations of hyperplanes : the groups $G_{25}$, $G_{26}$ are related to the symmetry group of
the so-called Hessian configuration of the nine inflection points of nonsingular cubic curves, while  $G_{32}$ acts by automorphisms on the configuration of
27 lines on a nonsingular cubic surface. These groups belong to the special case of so-called `Shephard groups', namely the symmetry groups
of regular complex polytopes. The group $G_{31}$, introduced
by H. Maschke in his first paper \cite{MASCH}, is not a Shephard group, has all its reflections of order $2$, and is connected
to the theory of hyperelliptic functions. It is the only `exceptional' reflection group in dimension $n \geq 3$ which cannot be generated by $n$ reflections.

\section{Artin groups and Paris representation}
\subsection{Preliminaries on Artin groups}

Let $S$ be a finite set. Recall that a \emph{Coxeter matrix} based on $S$ is a matrix $M = (m_{s,t})_{s,t \in S }$
indexed by elements of $S$ such that
\begin{itemize}
\item $m_{ss} = 1$ for all $s \in S$
\item $m_{st} = m_{ts} \in \{ 2,3 , \dots, \infty \}$ for all $s,t \in S$, $s \neq t$.
\end{itemize}
and that the Coxeter system associated to $M$ is the couple $(W,S)$, with $W$ the group
presented by $< S \ | \ \forall s \in S \ s^2 = 1, \forall s,t \in S \ \ (st)^{m_{st}} = 1 >$.
Let $\Sigma = \{ \sigma_s , s \in S \}$ be a set in natural bijection with $S$. The \emph{Artin system}
associated to $M$ is the pair $(B, \Sigma)$ where $B$ is the group presented by
$< \Sigma \ | \ \forall s,t \in S \ \ \underbrace{\sigma_s \sigma_t \sigma_s \dots}_{m_{s,t} \mbox{\ terms }} 
= \underbrace{\sigma_t \sigma_s \sigma_t \dots}_{m_{s,t} \mbox{\ terms }}>$, and called the Artin group associated to $M$. The Artin monoid $B^+$
is the monoid with the same presentation. According to \cite{PARIS}, the natural monoid morphism $\sigma_s \mapsto
\sigma_s$ , $B^+ \to B$, is an embedding. There is a natural morphism $B \onto W$ given by
$\sigma_s \mapsto s$, whose kernel is known as the pure Artin group $P$.

For the sequel we will need a slightly more specialized vocabulary, borrowed from \cite{PARIS}. A Coxeter matrix is said to be \emph{small}
if $m_{s,t} \in \{ 2, 3 \}$ for all $ s \neq t$, and it is called \emph{triangle-free} if there is no
triple $(s,t,r)$ in $S$ such that $m_{s,t}$, $m_{t,r}$ and $m_{r,s}$ are all greater than $2$.

\subsection{Paris representation}

To a Coxeter system as above is naturally associated a linear representation
of $W$, known as the reflection representation. We briefly recall
its construction. Let $\Pi = \{ \alpha_s ; s \in S \}$ denote
a set in natural bijection with $S$, called the set of simple roots.
Let $U$ denote the $\R$-vector space with basis $\Pi$, and
$< \ , \ > : U \times U \to \R$ the symmetric bilinear form defined by
$$
<\alpha_s,\alpha_t > = \left\lbrace \begin{array}{lcl} -2 \cos \left( \frac{\pi}{m_{st}} \right)
& \mbox{ if } & m_{st} < \infty \\
-2 & \mbox{otherwise} \end{array} \right.
$$
In particular $< \alpha_s,\alpha_s > = 2$. There is a faithful
representation $W \to \GL(U)$ defined by
$s(x) = x - < \alpha_s,x> \alpha_s$ for $x \in U$, $s \in S$, which
preserves the bilinear form $< \ , \ >$. Let $\Phi = \{ w \alpha_s ; s \in S, w \in W \}$
be the root system associated to $W$, $\Phi^+ = \{ \sum_{s \in S} \la_s \alpha_s \in \Phi ; \forall s \in S \ \ \la_s \geq 0 \}$,
and $\Phi^- = - \Phi^+$. We let $\ell$ denote the length function on $W$ (resp. $B^+$)
with respect to $S$ (resp. $\Sigma$). The \emph{depth} of $\beta \in \Phi^+$
is
$$
dp(\beta) = \min \{ m \in \N \ | \ \exists w \in W \ \ w. \beta \in \Phi^- \mbox{\ and\ } \ell(w) = m \}.
$$
We have (see \cite{PARIS} lemma 2.5) 
$$
dp(\beta) = \min \{ m \in \N \ | \ \exists w \in W , s \in S \ \ \beta = w^{-1}. \alpha_s \mbox{\ and\ } \ell(w)+1 = m \}.
$$
When $s \in S$ and $\beta \in \Phi^+ \setminus \{ \alpha_s \}$, we have
$$
dp(s.\beta) = \left\lbrace \begin{array}{lcl}
dp(\beta) - 1 & \mbox{ if } & \langle \alpha_s, \beta \rangle > 0 \\
dp(\beta) &  \mbox{ if } & \langle \alpha_s,\beta \rangle = 0 \\
dp(\beta) +1 & \mbox{ if } & \langle \alpha_s, \beta \rangle < 0 
\end{array}
\right.
$$

In \cite{PARIS}, polynomials $T(s,\beta) \in \Q[y]$ are defined for $s \in S$ and $\beta \in \Phi^+$.
They are constructed by induction on $dp(\beta)$, by the following formulas.
When $dp(\beta) = 1$, that is $\beta = \alpha_t$ for some $t \in S$, then
$$
\begin{array}{llcll}
\mathrm{(D1)} & T(s,\alpha_t) &=& y^2& \mbox{ if \ } t=s \\
\mathrm{(D2)}  & T(s,\alpha_t) &=& 0& \mbox{ if \ } t\neq s \\
\end{array}
$$
When $dp(\beta) \geq 2$, then there exists $t \in S$ such that
$dp(t. \beta) = dp(\beta) -1$, and we necessarily have $b = \langle \alpha_t,\betaÊ\rangle > 0$. In case $\langle \alpha_s, \beta \rangle > 0$,
we have
$$
\mathrm{(D3)}\ \  T(s,\beta) = y^{dp(\beta)}(y-1) ;
$$
in case $\langle \alpha_s,\beta \rangle = 0$, we have
$$
\begin{array}{llcll}
\mathrm{(D4)} & T(s,\beta) &=& yT(s,\beta - b \alpha_t) & \mbox{ if \ } \langle \alpha_s , \alpha_t \rangle = 0 \\
\mathrm{(D5)}& T(s,\beta) &=& (y-1) T(s,\beta- b \alpha_t) + y T(t,\beta- b \alpha_s - b \alpha_t) & \mbox{ if \ } \langle \alpha_s , \alpha_t \rangle = -1; \\
\end{array}
$$
and, in case $\langle \alpha_s, \beta \rangle = -a < 0$, we have

$$
\begin{array}{llcll}
\mathrm{(D6)} & T(s,\beta) &=& yT(s,\beta - b \alpha_t) & \mbox{ if \ } \langle \alpha_s , \alpha_t \rangle = 0 \\
\mathrm{(D7)}& T(s,\beta) &=& (y-1) T(s,\beta- b \alpha_t) + y T(t,\beta- (b-a) \alpha_s - b \alpha_t) & \mbox{ if \ } \langle \alpha_s , \alpha_t \rangle = -1 \mbox{ and } b > a \\
\mathrm{(D8)}& T(s,\beta) &=& T(t,\beta- b \alpha_t) + (y-1) T(s,\beta-b \alpha_t) & \mbox{ if \ } \langle \alpha_s , \alpha_t \rangle = -1  \mbox{ and } b = a \\
\mathrm{(D9)}& T(s,\beta) &=& y T(s,\beta- b \alpha_t) +  T(t,\beta- b \alpha_t)+y^{dp(\beta)-1}(1-y) & \mbox{ if \ } \langle \alpha_s , \alpha_t \rangle = -1   \mbox{ and } b < a. \\
\end{array}
$$

Now introduce $\mathcal{E} = \{ e_{\beta}; \beta \in \Phi^+ \}$
a set in natural bijection with $\Phi^+$, and let $V$ denote
the free $\Q[x,y,x^{-1},y^{-1}]$-module with basis $\mathcal{E}$.
For $s \in S$, one defines a linear map
$\varphi_s : V \to V$ by
$$
\begin{array}{lcll}
\varphi_s (e_{\beta}) &=& 0 & \mbox{ if \ } \beta = \alpha_s \\
& & e_{\beta}  & \mbox{ if \ }  \langle \alpha_s, \beta \rangle = 0 \\
 & & y e_{\beta - a \alpha_s }  & \mbox{ if \ } \langle \alpha_s , \beta \rangle = a > 0 \mbox{ and } \beta \neq \alpha_s \\
& & (1-y)e_{\beta} + e_{\beta+ a \alpha_s }  & \mbox{ if \ } \langle \alpha_s, \beta \rangle= -a < 0
\end{array}
$$
We have $\varphi_s \varphi_t = \varphi_t \varphi_s$ if $m_{s,t} = 2$, $\varphi_s \varphi_t \varphi_s = \varphi_t \varphi_s \varphi_t$
if $m_{s,t} = 3$. Now the Paris representation $\Psi : B \to \GL(V)$
is defined by $\Psi : \sigma_s \mapsto \psi_s$, with
$$
\psi_s(e_{\beta}) = \varphi_s(e_{\beta}) + x T(s,\beta) e_{\alpha_s }.
$$

\subsection{Reduction modulo $h$}
We embed $\Q[y]$ inside $\Q[[h]]$ under
$y \mapsto e^h$ and consider congruences $\equiv$ modulo $h$. 
Using the formulas of \cite{PARIS}, we deduce the
main technical step of our proof.

\begin{proposition} \label{propcruc}
Let $s \in S$ and $\beta \in \Phi^+$. Then $T(s,\beta) \equiv 1$ if $\beta = \alpha_s$
and $T(s,\beta) \equiv 0$ otherwise.
\end{proposition}
\begin{proof}
The case $dp(\beta) = 1$ is a consequence of (D1),(D2), as $y \equiv 1 \mod h$.
We thus browse through the various cases when $dp(\beta) \geq 2$, and use induction
on the depth. As in the definition of the polynomials, let $t \in S$
such that $dp(\gamma) = dp(\beta) -1$ for $\gamma = t. \beta$, and recall
that necessarily $b = \langle \alpha_t, \beta \rangle > 0$. In case
$\langle \alpha_s,\beta \rangle > 0$ then (D3) implies $T(s,\beta) \equiv 0$.
If $\langle \alpha_s, \beta \rangle = 0$, we have several subcases. If
$\langle \alpha_s, \alpha_t \rangle = 0$, then (D4) implies
$T(s,\beta) \equiv T(s,\gamma)$ with $\gamma = \beta - b  \alpha_t = t. \beta$
hence $dp(\gamma) < dp(\beta)$ and $T(s,\gamma) \equiv 0$ by induction,
unless $\gamma = \alpha_s$, that is $\alpha_s = \beta - b \alpha_t$, hence taking
the scalar product by $\alpha_s$ we would get $2 = 0$, a contradiction.
Otherwise, we have $\langle \alpha_s , \alpha_t \rangle = -1$. In that case,
(D5) implies $T(s,\beta) \equiv T(t, \beta- b \alpha_s - b \alpha_t)$. Note
that $\beta - b \alpha_t  = \gamma = t. \beta$, $\langle \alpha_s, \gamma \rangle
 = 0 - b \langle \alpha_t, \alpha_s \rangle = b$ and $s.\gamma = \gamma - b \alpha_s$.
Thus $T(s,\beta) \equiv T(t,st.\beta)$. Now $\langle \alpha_s,\gamma \rangle = b > 0$ hence
$dp(s.\gamma) = dp(\gamma) -1 < dp(\beta)$, unless $\gamma = \alpha_s$ ; but the case
$\gamma = \alpha_s$ cannot occur here, as it would imply
$$
2 = \langle \alpha_s, \alpha_s \rangle = \langle \alpha_s,\gamma \rangle = 
\langle \alpha_s,t.\beta \rangle = 
\langle \alpha_s,\beta - b \alpha_t  \rangle = 
-b \langle \alpha_s,  \alpha_t  \rangle = b$$
hence $\alpha_s = \gamma = t.\beta = \beta - 2 \alpha_t$, whence
$-1 = \langle \alpha_s, \alpha_t \rangle = 
\langle \beta, \alpha_t \rangle -2 \langle \alpha_t,\alpha_t \rangle = b - 4 = -2$,
a contradiction.
 
Finally,
$T(t,st.\beta) \equiv 0$ by induction unless $\beta - b \alpha_s -b \alpha_t = \alpha_t$,
in which case scalar product by $\alpha_t$ leads to the contradiction $2=0$.

The last case is when $\langle \alpha_s, \beta \rangle = -a < 0$, which is
subdivided in 4 subcases. Either $\langle \alpha_s, \alpha_t \rangle = 0$,
and then (D6) implies $T(s,\beta) \equiv T(s,\beta - b \alpha_t) \equiv 0$,
as $\beta - b \alpha_t = \alpha_s$ cannot occur (scalar product with $\alpha_s$ yields $2 = -a < 0$).
Or $\langle \alpha_s, \alpha_t \rangle = -1$ and $b > a$, then (D7) implies
$T(s,\beta) \equiv T(t,\beta - (b-a) \alpha_s - b \alpha_t) \equiv 0$
unless $\alpha_t = \beta - (b-a) \alpha_s - b \alpha_t$, which cannot occur for
the same reason as before (take the scalar product with $\alpha_t$). Or,
$\langle \alpha_s , \alpha_t \rangle = -1$ and $b=a$, in
which case (D8) implies $T(s,\beta) \equiv T(t,\beta- b \alpha_t) \equiv 0$,
as $\alpha_t \neq \beta - b \alpha_t$ (take the scalar product with $\alpha_t$).
Finally, the last subcase is $\langle \alpha_s, \alpha_t \rangle = -1$ and $b < a$,
then (D9) implies $T(s, \beta) \equiv T(s,\beta- b \alpha_t) + T(t,\beta - b \alpha_t) \equiv 0$,
unless $\alpha_s = \beta - b \alpha_t$, which leads to the contradiction $2 = b-a < 0$
under $\langle \alpha_s, \cdot \rangle$, or $\alpha_t = \beta - b \alpha_t$,
which leads to the contradiction $2 = -b < 0$
under $\langle \alpha_t, \cdot \rangle$. 
\end{proof}

We now embed $\Q[x^{\pm 1}, y^{\pm 1}]$ into $\Q(\sqrt{2})[[h]]$ under $y \mapsto e^h$,
$x \mapsto e^{\sqrt{2} h}$ (any other irrational than $\sqrt{2}$ would also do),
and define 
$\tilde{V} = V \otimes_{\iota} \Q(\sqrt{2})[[h]]$ where $\iota$
is the chosen embedding; that is, $\tilde{V}$
is the free $\Q(\sqrt{2})[[h]]$-module with basis $\mathcal{E}$, and clearly
$V \subset \tilde{V}$. We similarly introduce the $\Q(\sqrt{2})$-vector space $V_0$ with basis $\mathcal{E}$.
One has $\GL(V) \subset \GL(\tilde{V})$, and a reduction morphism $\End(\tilde{V}) \to \End(V_0)$.
Composing both we get elements $\overline{\psi_s}, \overline{\varphi_s} \in \End(V_0)$ associated to the $\psi_s \in \GL(V)$,
$\varphi_s \in \End(V)$. Because $x \equiv 1 \mod h$ and because of proposition \ref{propcruc} one gets from the definition of $\psi_s$ that
$$
\begin{array}{lcll}
\overline{\psi}_s(e_{\beta}) &=& \overline{\varphi_s}(e_{\beta}) & \mbox{ if\ } \beta \neq \alpha_s \\
& =  & \overline{\varphi_s}(e_{\alpha_s}) + e_{\alpha_s} = e_{\alpha_s} & \mbox{ if\ } \beta = \alpha_s. \\
\end{array}
$$
We denote $(w,\beta) \mapsto w \star \beta$ the natural action of $W$ on $\Phi^+$,
that is $w \star \beta = \beta$ if $w. \beta \in \Phi^+$, $w \star \beta = - \beta \in \Phi^+$ if $w.\beta \in \Phi^-$. The previous
equalities imply
$$
\forall s \in S \ \ \forall \beta \in \Phi^+ \ \ \overline{\psi_s}(e_{\beta}) = e_{s \star \beta}
$$
From this we deduce the following.

\begin{proposition}
For all $g \in P$, $\overline{\Psi(g)} = \mathrm{Id}_{V_0}$.
\end{proposition}

\begin{proof}
Recall that $P$ is defined as $\Ker (\pi : B \onto W )$. From
$\overline{\psi_s}(e_{\beta}) = e_{s \star \beta}$
one gets $\overline{\Psi(g)}(e_{\beta}) = e_{\pi(g) \star \beta}$
for all $g \in B$
and the conclusion.
\end{proof}

As a consequence $\Psi(P) \subset \{ \varphi \in \GL(V) \ | \ \overline{\varphi} = \mathrm{Id}_{V_0} \}$.
The group
$\{ \varphi \in \GL(V) \ | \overline{\varphi} =  \mathrm{Id}_{V_0} \}$ is a subgroup of
$G = \{ \varphi \in \GL(\tilde{V}) \ | \overline{\varphi} =  \mathrm{Id}_{V_0} \}$, which we now prove to
be residually torsion-free nilpotent. We adapt the argument of lemma \ref{lemresnil} to the infinite-dimensional case. Let $\kk = \Q(\sqrt{2})$. The
canonical projection $\kk[[h]] \onto \kk[h]/h^r$ extends to a morphism $\pi_r : \End(\tilde{V}) \to
\End(V_0)\otimes_{\kk} \kk[h]/h^r$ with clearly $\pi_1(\varphi) = \overline{\varphi}$.
Let $G_r = \{ \varphi \in \GL(\tilde{V}) \ | \ \pi_r(\varphi) = \mathrm{Id} \}$.
Then $G/G_r$ is identified to $\pi_r(G)$ which is a subgroup of $\{ \mathrm{Id}_{V_0} + h u \ | \ u \in \End(V_0) \otimes_{\kk} \kk[h]/h^r \}$,
which is clearly torsion-free and nilpotent. Since $\bigcap_r G_r = \{ 1 \}$ this proves the residual torsion-free nilpotence of
$G$ and
theorem \ref{theoA}.

\section{Braid groups of complex reflection groups}

A pseudo-reflection in $\C^n$ is an endomorphism which fixes an hyperplane.
For $W$ a finite subgroup of $\GL_n(\C)$ generated by pseudo-reflections
(so-called complex reflection group), we have
a \emph{reflection
arrangement} $\mathcal{A} = \{ \Ker (s-1) | s \in \mathcal{R} \}$, where
$\mathcal{R}$ is the set of reflections of $W$. Letting $X$ denote the
hyperplane complement $X = \C^n \setminus \bigcup \mathcal{A}$, the fundamental
groups $P = \pi_1(X)$ and $B = \pi_1(X/W)$ are called the pure braid group and braid
group associated to $W$. The case of spherical type Artin groups corresponds to the
case where $W$ is a finite Coxeter group.

It is conjectured that $P$ is always residually torsion-free nilpotent. For
this one can assume that $W$ is irreducible. According to \cite{SHEPTODD}, such a $W$ belongs either to
an infinite series $G(de,e,n)$ depending on three integer parameters
$d,e,n$, or to a finite set of 34 exceptions, denoted $G_4,\dots,G_{37}$. The fiber-type argument of \cite{FALK,FALKDEUX}
to prove the residual torsion-free nilpotence only works
for the groups $G(d,1,n)$, and when $n = 2$.

For the case of $W$ a finite Coxeter group, we used the Krammer representation
to prove that $P$ is residually torsion-free nilpotent in \cite{RESNIL,KRAMINF}.
The exceptional groups of rank $n > 2$ which are not Coxeter groups are
the 9 groups $G_{24}$, $G_{25}$, $G_{26}$, $G_{27}$, $G_{29}$, $G_{31}$, $G_{32}$, $G_{33}$, $G_{34}$.

We show here that this argument can be adjusted to prove the
residual torsion-free nilpotence for a few of them.

We begin with the Shephard groups $G_{25}$, $G_{26}$, $G_{32}$.
The Coxeter-like diagrams of these groups are the following ones.

\def\nnode#1{{\kern -0.6pt\mathop\bigcirc\limits_{#1}\kern -1pt}}
\def\ncnode#1#2{{\kern -0.4pt\mathop\bigcirc\limits_{#2}\kern-8.6pt{\scriptstyle#1}\kern 2.3pt}}
\def\sbar#1pt{{\vrule width#1pt height3pt depth-2pt}}
\def\dbar#1pt{{\rlap{\vrule width#1pt height2pt depth-1pt} 
                 \vrule width#1pt height4pt depth-3pt}}

$$
G_{25}\ \ \  \ncnode3s\sbar16pt\ncnode3t\sbar16pt\ncnode3u\ \ \ \ \ \ \ \ \ \ \ \ 
G_{26}\ \ \  \ncnode2s\dbar16pt\ncnode3t\sbar16pt\ncnode3u \ \ \ \ \ \ \ \ \ \ \ \ 
G_{32}\ \ \  \ncnode3s\sbar10pt\ncnode3t\sbar10pt\ncnode3u\sbar10pt\ncnode3v
$$

It is known (see \cite{BMR}) that removing the conditions on the order
of the generators gives a (diagrammatic) presentation of the corresponding braid group.
In particular, these have for braid groups the Artin groups of
Coxeter type $A_3, B_3$ and $A_4$, respectively.

We recall a matrix expression of the Krammer representation
for $B$ of Coxeter type $A_{n-1}$, namely for the classical braid group on $n$
strands. Letting $\sigma_1,\dots,\sigma_{n-1}$ denote its Artin generators
with relations $\sigma_i \sigma_j = \sigma_j \sigma_i$ if $|j - i| \geq 2$,
$\sigma_i \sigma_{i+1} \sigma_i = \sigma_{i+1} \sigma_i \sigma_{i+1}$, their action
on a specific basis $x_{ij}$ ($1 \leq i < j \leq n$) is given by the following formulas (see \cite{KRAM})
$$
\left\lbrace \begin{array}{ll}
\sigma_k x_{k,k+1} = tq^2 x_{k,k+1} \\
\sigma_k x_{i,k} = (1-q)x_{i,k} + q x_{i,k+1} & i<k\\
\sigma_k x_{i,k+1} = x_{i,k} + t q^{k-i+1} (q-1) x_{k,k+1} & i<k \\
\sigma_k x_{k,j} = tq(q-1) x_{k,k+1} + q x_{k+1,j}& k+1<j \\
\sigma_k x_{k+1,j} = x_{k,j} + (1-q) x_{k+1,j} & k+1<j\\
\sigma_k x_{i,j} = x_{i,j} & i<j<k \mbox{ or } k+1<i<j \\
\sigma_k x_{i,j} = x_{i,j} + tq^{k-i} (q-1)^2  x_{k,k+1} & i<k<k+1<j\\
\end{array} \right.
$$
where $t$ and $q$ denote algebraically independent parameters. We
embed the field $\Q(q,t)$ of rational fractions in $q,t$ into $K = \C((h))$ by $q \mapsto -\zeta_3 e^h$ and
$t \mapsto e^{\sqrt{2}h}$, where $\zeta_3$ denotes a primitive 3-root of 1.
We then check by an easy calculation that $\sigma_k^3 \equiv 1$
modulo $h$. Since the quotients of the braid group on $n$
strands by the relations $\sigma_k^3 = 1$
are, for $n = 3,4,5$, the Shephard group of types $G_4, G_{25}$ and $G_{32}$,
respectively, it follows that the pure braid groups of these types
embed in 
$\GL_N^0(A)$ with
$N = n(n-1)/2$ and $\kk = \C$, and this proves their residual torsion-free nilpotence
by lemma \ref{lemresnil}.

We now turn to type $G_{26}$. Types $G_{25}$ and $G_{26}$
are symmetry groups of regular complex polytopes which are
known to be closely connected (for instance they both appear in the
study of the Hessian configuration, see e.g. \cite{COXETER} \S 12.4 and
\cite{ORLIKTERAO} example 6.30). The hyperplane arrangement
of type $G_{26}$ contains the 12 hyperplanes of type $G_{25}$
plus 9 additional ones. The natural inclusion induces morphisms
between the corresponding pure braid groups, which cannot
be injective, since a loop around one of the extra
hyperplanes is non trivial in type $G_{26}$.
However we will prove the
following, which proves the residual torsion-free nilpotence in type $G_{26}$.

\begin{proposition} \label{prop2526}The pure braid group of type $G_{26}$ embeds into
the pure braid group of type $G_{25}$.
\end{proposition}

More precisely, letting $B_i,P_i,W_i$ denote the braid group, pure
braid group and pseudo-reflection group of type $G_i$, respectively, we
construct morphisms $B_{26} \into B_{25}$ and $W_{26} \onto W_{25}$
such that the following diagram commutes, where the vertical arrows
are the natural projections.
$$
\xymatrix{
B_{25} \ar@{->>}[d] & B_{26}\ar@{_{(}->}[l] \ar@{->>}[d]\\
W_{25} & \ar@{->>}[l]W_{26}
}
$$
Both horizontal morphisms are given by the formula
$(s,t,u) \mapsto ((tu)^3,s,t)$, where $s,t,u$ denote the generators
of the corresponding groups according to the above diagrams. The morphism between
the pseudo-reflection groups is surjective because it is a retraction
of an embedding $W_{25} \into W_{26}$
mapping $(s,t,u)$ to $(t,u,t^{sut^{-1}u})$. 
The kernel of this projection is the subgroup of order
2 in the center of $W_{26}$ (which has order 6).

We now consider the morphism between braid groups and prove
that it is injective. First recall that the braid group
of type $G_{26}$ can be identified with the Artin group of
type $B_3$. On the other hand, Artin groups of type $B_n$
are isomorphic to the semidirect product of the Artin group
of type $A_{n-1}$, that we denote $\mathcal{B}_n$
to avoid confusions,
with a free group $F_n$ on $n$ generators $g_1,\dots,g_n$,
where the action (so-called `Artin action') is given (on the left) by
$$\sigma_i :
\left\lbrace 
\begin{array}{lcll}
g_i &\mapsto &g_{i+1}\\
g_{i+1} &\mapsto & g_{i+1}^{-1} g_i g_{i+1} \\
g_j &\mapsto & g_j & \mbox{ if } j \not\in \{i,i+1 \}
\end{array} \right.  
$$
If $\tau, \sigma_1,\dots,\sigma_{n-1}$ are the standard generators of the Artin
group of type $B_n$, with $\tau \sigma_1 \tau \sigma_1 = \sigma_1 \tau \sigma_1 \tau$,
$\tau \sigma_i = \sigma_i \tau$ for $i > 1$, and usual braid relations
between the $\sigma_i$, then this isomorphism is given by
$\tau \mapsto g_1$,
$\sigma_i \mapsto \sigma_i$ (see \cite{CRISPPARIS} prop. 2.1 (2) for more details).
Finally, there exists an embedding of this semidirect product
into the Artin group $\mathcal{B}_{n+1}$ of type $A_{n}$ which satisfies $g_1 \mapsto
 (\sigma_2 \dots \sigma_{n})^n$, and $\sigma_i \mapsto \ \sigma_i$
($i \leq n-1$). By composing both, we get an embedding which makes
the square commute. This proves proposition \ref{prop2526}.

This embedding of type $B_n$ into type $A_n$,
different from the more standard
one $\tau \mapsto \sigma_1^2, \sigma_i \mapsto \sigma_{i+1}$,
has been considered in \cite{LONG}. The algebraic proof given
there being somewhat sketchy, we provide the details here. This
embedding comes from the following construction.

Consider the (faithful)
Artin action as a morphism $\mathcal{B}_{n+1} \to
\Aut(F_{n+1})$, and the free subgroup $F_{n} = \langle g_1,\dots,g_{n}\rangle$ of $F_{n+1}$.
The action of $\mathcal{B}_{n+1}$
preserves the product $g_1 g_2\dots g_{n+1}$, and there is a natural
retraction $F_{n+1} \onto F_{n}$ which sends $g_{n+1}$ to $(g_1 \dots g_{n})^{-1}$.
This induces a map $\Psi : \mathcal{B}_{n+1} \to \Aut(F_{n})$, whose kernel is the center of
$\mathcal{B}_{n+1}$
by a theorem of Magnus (see \cite{MAGNUS}).
We claim that its image contains the group $\mathrm{Inn}(F_n)$ of inner automorphisms
of $F_{n}$, which is naturally isomorphic to $F_{n}$.

Indeed, it is straightforward to check
that $b_1 = (\sigma_2 \dots \sigma_{n})^{n}$
is mapped to $\Ad(g_1) = x \mapsto g_1 x g_1^{-1}$. Defining
$b_{i+1} = \sigma_i b_i \sigma_i^{-1}$, we get that
$b_i$ is mapped to $\Ad(g_i)$. In particular the subgroup
$\mathcal{F}_{n} = \langle b_1,\dots,b_{n}\rangle $ of $\mathcal{B}_{n+1}$ is free and there is a natural
isomorphism $\varphi : b_i \mapsto g_i$ to $F_{n}$
characterized by the property $b.g = \Ad(\varphi(b))(g)$ for all $g \in F_n$
and $b \in \mathcal{F}_n$, that is $\Ad(\varphi(b)) = \Psi(b)$ for all $b \in \mathcal{F}_n$.

Now,
let $\mathcal{B}_n \subset \mathcal{B}_{n+1}$ be generated by
$\sigma_i, i \leq n-1$. Its action on $F_n$ is the usual Artin action
recalled above.
For $\sigma \in \mathcal{B}_n$ and $b \in \mathcal{F}_{n}$ we know
that $$
\forall x \in F_n \ \ \ \sigma b \sigma^{-1} . x = \sigma . \left( \varphi(b) (\sigma^{-1}.x) \varphi(b)^{-1}) \right)
=  (\sigma .  \varphi(b)) x (\sigma.\varphi(b))^{-1}
$$
that is
$\sigma b \sigma^{-1}$ is mapped to $\Ad(\sigma.\varphi(b))$
in $\Aut(F_{n})$, hence $\sigma b \sigma^{-1}$ and
$\varphi^{-1}(\sigma.\varphi(b)) \in \mathcal{F}_{n}$
have the same image under $\Psi$.
Since
the kernel of $\Psi$
is $Z(\mathcal{B}_{n+1})$, this proves that they
may differ only by
an element of the center $Z(\mathcal{B}_{n+1})$ of $\mathcal{B}_{n+1}$.
On the other hand, $\varphi : \mathcal{F}_n \to F_n$
commutes with the maps $F_n \to \Z$
and $\eta : \mathcal{F}_{n} \to \Z$ which map every generator to 1.
Likewise, the
Artin action commutes with $F_n \to \Z$ hence
$\eta(\varphi^{-1}(\sigma.\varphi(b))) = \eta(b)$.
We denote $\ell : \mathcal{B}_{n+1} \to \Z$ the abelianization map. We have
$\ell(b_i) = n(n-1)$ for all $i$, hence $\ell(b) = n(n-1)\eta(b)$
for all $b \in \mathcal{F}_{n}$. Since $\ell(b) =
\ell(\sigma b \sigma^{-1})$
it follows that $\sigma b \sigma^{-1}$ and
$\varphi^{-1}(\sigma.\varphi(b)) \in \mathcal{F}_{n}$ differ by an element
in $Z(\mathcal{B}_{n+1}) \cap (\mathcal{B}_{n+1},\mathcal{B}_{n+1})$,
where $(\mathcal{B}_{n+1},\mathcal{B}_{n+1})$ denotes the commutators subgroup.
But $Z(\mathcal{B}_{n+1})$
is generated by $(\sigma_1 \dots \sigma_{n})^{n+1}
\not\in (\mathcal{B}_{n+1},\mathcal{B}_{n+1})$
hence $\sigma b \sigma^{-1}=\varphi^{-1}(\sigma.\varphi(b)) \in \mathcal{F}_{n}$.

In particular $\mathcal{F}_{n}$ is stable under the action by
conjugation of $\mathcal{B}_{n}$, which
coincides with the Artin action. This is the embedding $\mathcal{B}_n \ltimes F_n \into
\mathcal{B}_{n+1}$ that is needed
to make the square commute. It remains to prove that we indeed have a semidirect product,
namely that $\mathcal{B}_n \cap \mathcal{F}_n = \{ 1 \}$.
First notice that $\mathcal{F}_n$ is mapped to $\mathrm{Inn}(F_n)$ and recall that the outer
Artin action $\mathcal{B}_n \to \mathrm{Out}(F_n)$ has for kernel $Z(\mathcal{B}_n)$,
hence $\mathcal{F}_n \cap \mathcal{B}_n \subset Z(\mathcal{B}_n)$. Then $x \in \mathcal{F}_n \cap \mathcal{B}_n$
can be written $x = z^k$ for some $k \in \Z$ with $z = (\sigma_1 \dots \sigma_{n-1})^n$. It is classical and easy to
check that the action of
$z$ on $F_n$ is given by $\Ad( (g_1 \dots g_n)^{-1})$, hence $\varphi(z^k) = \varphi((b_1 \dots b_n)^{- k})$
and $x = z^k =  (b_1 \dots b_n)^{- k}$.
Thus $\ell(x) = k n(n-1) = -k n^2(n-1)$ hence $k = 0$ and $x = 1$.

This concludes the case of $G_{26}$.
The case of $G_{31}$ is a consequence of the lifting of Springer's theory of `regular elements' for complex
reflection groups to their associated braid group.  By Springer theory (see \cite{SPRINGER}), 
$W_{31}$ appears as the centralizer of a regular element $c$ of order $4$ in $W_{37}$, which is the Coxeter group of type $E_8$, and, 
as a consequence of \cite[thm. 12.5 (iii)]{BESSIS},
$B_{31}$ can be identified with the centralizer of a lift
$\tilde{c} \in B_{37}$ of $c$, in such a way that the natural diagram
$$
\xymatrix{
B_{31}\ar@{^{(}->}[r]\ar@{->>}[d] &   B_{37}\ar@{->>}[d] \\
W_{31}\ar@{^{(}->}[r]  &   W_{37} \\
}$$
commutes. This embedding $B_{31} \into B_{37}$ is explicitly described in \cite{DMM}, to which we refer for more details.
By commutation of the above diagram it induces an embedding $P_{31} \into P_{37}$. Since $P_{37}$ is known to be residually torsion-free nilpotent by \cite{RESNIL,KRAMINF},
this concludes the proof of theorem \ref{theoB}.

\end{document}